\documentclass[11pt]{article}

\usepackage{amsmath,amsthm,amsfonts,amssymb,amscd}
\usepackage[english]{babel}
\usepackage[a4paper]{geometry}
\usepackage{graphicx}
\usepackage{float}
\usepackage{xcolor}
\usepackage[hidelinks]{hyperref}
\usepackage[T1]{fontenc}
\usepackage{lmodern}
\usepackage{mathtools}
\usepackage[hidelinks]{hyperref}

\newcommand{\ee}{\varepsilon}
\newcommand{\pa}{\partial}

\newtheorem{theorem}{Theorem}

\newtheorem{remark}{Remark}

\renewcommand{\tilde}{\widetilde}

\usepackage{enumerate}

\usepackage{authblk}

\begin{document}
	
	\title{New unexpected limit operators for homogenizing optimal control
		parabolic problems with dynamic reaction flow on the boundary of\\
		critically scaled particles}
	\date{}

	\author[1]{Jes\'us Ildefonso D\'iaz\thanks{Corresponding Author: \url{jidiaz@ucm.es}}}
	
	\affil[1]{%
		Instituto de Matem\'{a}tica Interdisciplinar, 
		Universidad Complutense de
		Madrid. \protect\\
		Plaza de Ciencias 3, 28040 (Madrid) Spain.}
	
	\author[2]{Alexander V. Podolskiy\thanks{\url{originalea@ya.ru}}}
	\affil[2]{%
		Faculty of Mechanics and Mathematics, 
		Moscow State University.
		Moscow 19992, Russia.}
	
	\author[2]{Tatiana A. Shaposhnikova\thanks{\url{shaposh.tan@mail.ru}}}

\maketitle

\begin{abstract}
We study the asymptotic behavior, as $\varepsilon \rightarrow 0,$ of the
optimal control and the optimal state of an initial boundary value problem
in a domain that is $\varepsilon $-periodically perforated by balls (or,
equivalently it is the complementary to a set of spherical particles). On
the boundary of the perforations (or of the particles) we assume a dynamic
condition with a large growth coefficient in the time derivative. The
control region is a possible small subregion and the cost functional
includes a balance between the prize of the controls and the error with
respect to a given target profile. We consider the so-called
\textquotedblleft critical case\textquotedblright\ concerning a certain
relation between the structure's period, the diameter of the balls, and the
growth coefficient of the boundary condition.  We show that the homogenized
problem contains in the limit state equation a nonlocal "strange term",
given as a solution to a suitable ordinary differential equation. We prove
the weak convergence of the state and the optimal control to the state and
the optimal control associated with the limit cost functional \ which now
contains an unexpected new \textquotedblleft strange\textquotedblright\ term.
\end{abstract}


\section{Introduction}
It is well-known that important problems of Chemical Engineering lead to the optimization of some cost functionals (see, e.g. \cite{Upreti book} and
the survey \cite{Nolasco-Survey}). Here, we will consider the optimal control
problem associated with the distributed case in which the chemical reactor
consists of a fixed bed and a dynamic reaction flow on the boundary
of the particles. The problem also arises in other frameworks linked to
porous media in which the word \textquotedblleft particle\textquotedblright\
must be replaced by \textquotedblleft perforation\textquotedblright\ (see,
e.g. \cite{Hurlov}, \cite{Hornung-Yaeger}, \cite{ConcaDliTi}, \cite{OlSh95}, 
\cite{GoSanchezP}, \cite{Iliev-Mikelik}, \cite{Timof}, \cite{Anguiano} and
the many other references quoted in the monograph \cite{DiGoShBook}). The
state equation of our control problem is the parabolic problem with dynamic
boundary conditions%
\begin{equation}
\left\{ 
\begin{array}{lr}
\partial _{t}u_{\varepsilon }(v)-\Delta u_{\varepsilon }(v)=f+\chi _{\omega
_{\varepsilon }}v, & (x,t)\in Q_{\varepsilon }^{T}, \\ 
\varepsilon ^{-\gamma }\partial _{t}u_{\varepsilon }(v)+\partial _{\nu
}u_{\varepsilon }(v)=0, & (x,t)\in S_{\varepsilon }^{T}, \\ 
u_{\varepsilon }(v)(x,0)=0, & x\in \Omega _{\varepsilon }, \\ 
u_{\varepsilon }(v)(x,0)=0, & x\in S_{\varepsilon }, \\ 
u_{\varepsilon }(v)(x,t)=0, & (x,t)\in \Gamma ^{T},%
\end{array}%
\right.  \label{init state prob}
\end{equation}%
where we are using the notation (considering $0 < T < \infty$)%
\begin{equation}\label{def: basic sets}
\begin{array}{cccc}
\Omega_{\varepsilon }=\Omega \setminus \overline{G_{\varepsilon }}, & 
S_{\varepsilon }=\partial G_{\varepsilon }, & \partial \Omega _{\varepsilon
}=S_{\varepsilon }\cup \partial \Omega , & \\ 
Q_{\varepsilon }^{T}=\Omega _{\varepsilon }\times (0,T), & \Gamma
^{T}=\partial \Omega \times (0,T), & S_{\varepsilon }^{T}=S_{\varepsilon
}\times (0,T), & Q^{T}=\Omega \times (0,T),%
\end{array}%
\end{equation}%
which will be detailed in the next section, $G_{\varepsilon }$ is  the set of small particles ( $\varepsilon $-periodically
distributed and homothetic to a ball) in an open bounded regular set $\Omega 
$ of $\mathbb{R}^{n}$, $n\geq 3$, $f\in L^2(Q^T)$, and the control $v\in L^{2}(\omega_{\varepsilon }^{T})$ is acting on a possibly small open part $\omega_{\varepsilon }$ of $\Omega_{\varepsilon }$ (here, $\chi_{\omega_{\varepsilon }}$ denotes the characteristic function of $\omega_{\varepsilon }$): i.e. we introduce an open domain $\omega$ such that $\overline\omega \subset
\Omega $, and then we consider the sets 
\begin{equation*}
\omega _{\varepsilon }=\omega \cap \Omega _{\varepsilon },\quad \omega
_{\varepsilon }^{T}=\omega _{\varepsilon }\times (0,T),\quad \omega
^{T}=\omega \times (0,T).
\end{equation*}%
\ The parameter $\gamma >0$ plays crucial role since in this paper
we will consider the so-called \textquotedblleft critical
case\textquotedblright\ in which each particle is a translation of a small
particle $a_{\varepsilon }G_{0},$ where $G_{0}$ is the unit ball and $%
a_{\varepsilon }=C_{0}\varepsilon ^{\gamma }$, with $\gamma =\frac{n}{n-2}$
and $C_{0}$ some positive constant.

Notice, since problem (\ref{init state prob}) is a linear problem, we can
assume without loss of generality that the initial data are zero. We also
assume $f\in L^{2}(Q^{T})$ and the control regularity $v\in L^{2}(\omega
_{\varepsilon }^{T})$. The existence of a unique weak solution $%
u_{\varepsilon }(v)\in L^{2}(0,T;H^{1}(\Omega _{\varepsilon }, \pa\Omega))$ with $%
\partial _{t}u_{\varepsilon }(v)\in L^{2}(0,T;L^{2}(\Omega _{\varepsilon }))$, $\pa_t u_\ee(v) \in L^2(0, T; L^2(S_\ee))$ 
can be obtained by the Galerkin approximations (see~\cite{ZuShDynDok19}). The application of the
abstract theory for subdifferential of convex functions also leads to some
existence results, see Remark \ref{rm Abstract theory}. 

The formulation of the optimal problem ends with the definition of the cost
functional $J_{\varepsilon }:L^{2}(\omega_{\varepsilon }^{T})\rightarrow 
\mathbb{R}$. We assume to be given a \textquotedblleft
target\textquotedblright\ function, a given profile observed at the final time $T$, $u_{T}\in H^{1}(0,T;H_{0}^{1}(\Omega)) \bigcap C(\overline{Q^T})$, and we try to optimize a
weighted functional making the balance between the \textquotedblleft
prize\textquotedblright\ of each control $v\in L^{2}(\omega _{\varepsilon
}^{T})$ and the \textquotedblleft error with respect to the target
function\textquotedblright\ at the final time $T$ (over $\Omega
_{\varepsilon }$ and over $S_{\varepsilon }$) but in the same time trying to have a
spatial gradient of the state $u_{\varepsilon }(v)(x,t)$ close to the one of
the target profile $u_{T}$ when $t\in (0,T)$ (notice that we know that the
trace of $u_{T}$ on $S_{\varepsilon }$ is well defined), i.e. we consider the
cost functional%
\begin{equation}
\begin{gathered}
J_{\varepsilon }(v)=\frac{1}{2}\int\limits_{Q_{\varepsilon }^{T}}|\nabla
u_{\varepsilon }(v)-\nabla u_{T}|^{2}{dxdt}+\frac{1}{2}\int\limits_{\Omega
_{\varepsilon }}(u_{\varepsilon }(v)-u_{T})^{2}(x,T){dx} \\ 
+\frac{{\varepsilon }^{-\gamma }}{2}\int\limits_{S_{\varepsilon
}}(u_{\varepsilon }(v)-u_{T})^{2}(x,T){ds}+\frac{N}{2}\int\limits_{\omega
_{\varepsilon }^{T}}v^{2}{dxdt}.%
\end{gathered}
\label{cost functional}
\end{equation}
Some comments on the application of this optimal control problem to the
study of the approximate controllability of solutions of (\ref{init state
prob}) will be given later (see Remark \ref{Rm Controlability probl init}).
By applying different results in the literature, see \cite{LionsOpt}, \cite{Fursikov2000OptimalCO}, \cite{Trolstz}, \cite{Glow-Lions}, it is well known
that there exists a unique optimal control $v_{\varepsilon }\in L^{2}(\omega
_{\varepsilon }^{T})$ such that 
\begin{equation}
J_{\varepsilon }(v_{\varepsilon })=\inf\limits_{v\in L^{2}(\omega
_{\varepsilon }^{T})}J_{\varepsilon }(v).  \label{optim control def}
\end{equation}

The main goal of this paper is to apply a homogenization process to the
above optimal control problem when $\varepsilon \rightarrow 0$. As in many
other formulations, the kind of homogenized limit problem depends,
strongly, on the size of the particles radii $C_{0}\varepsilon ^{\alpha }$, $%
C_{0}>0$ (see, e.g. \cite{Timof}, \cite{Anguiano} and the exposition made in 
\cite{DiGoShBook}). Here, we will consider the critical case in which $\alpha
=\gamma ={n}/{(n-2).}$ For some different elliptic and parabolic problems, it
is well-known that this critical choice leads to the emergence of a new
\textquotedblleft strange\textquotedblright\ term (the naming is due to~\cite%
{Cioranescu1997AST}) in the effective partial differential equation (see~%
\cite{Hurlov}, \cite{Cioranescu1997AST}, \cite{Kaizu1}, \cite{ZuShDiffEq},
and the monograph \cite{DiGoShBook}).

In the framework of elliptic equations with a dynamic boundary condition on
the boundary of the particles, it was shown that the above-mentioned
\textquotedblleft strange\textquotedblright\ term \ becomes a
\textquotedblleft non-local\textquotedblright\ operator obtained by solving
a suitable ordinary differential equation (we refer to \cite{ShDyn19,
TimeDepNon20}). In this paper, we will show that a new unexpected term
appears in the limit cost functional (in contrast with previous results in
the literature for related formulations, e.g. see~\cite{SJPZ02,
PoShControl20, Shaposhnikova2022HomogenizationOT, DiPoShBoundCont22}).
Although, the detailed statements of our results will be presented later, we
summarize now that\ we will prove the convergence of the optimal
controls $v_{\varepsilon }$ $\rightarrow v_{0}$ strongly in $L^{2}(\omega
^{T})$, the convergence of the corresponding states (extended to $\Omega $) 
$\tilde{u}_{\varepsilon }\rightharpoonup u_{0}$ weakly in $%
L^{2}(0,T;H_{0}^{1}(\Omega, \pa\Omega))$ and $\pa_t\tilde{u}_{\varepsilon}\rightharpoonup \pa_t u_{0}$ weakly in $L^{2}(Q^T)$, with the limit state problem given by 
\begin{equation}\label{limit u prob}
\left\{ 
\begin{array}{lr}
\partial _{t}u_{0}-\Delta u_{0}+\mathcal{A}_{n}(u_{0}-\mathcal{B}%
_{n}H(u_{0})(x,t))=f + \chi_{\omega }v_{0}, & (x,t)\in Q^{T}, \\ 
u_{0}(x,0)=0, & x\in \Omega , \\ 
u_{0}(x,t)=0, & (x,t)\in \Gamma ^{T},%
\end{array}%
\right.
\end{equation}%
for suitable constants $\mathcal{A}_{n},\mathcal{B}_{n}$ and a suitable
non-local in time operator $H(u_0)$, and that $v_{0}$ is the optimal
control associated to the limit cost functional $J_{0}(v)$ (i.e. 
$J_0(v_0) = \inf\{J_0(v) \vert v\in L^2(\omega^T)\}$
)
that is defined by
\begin{equation}\label{limit cost functional}
\begin{gathered}
J_{0}(v)=\frac{1}{2}\Vert \nabla (u_{0}(v)-u_{T})\Vert _{L^{2}(Q^{T})}^{2}+%
\frac{1}{2}\Vert u_{0}(v)(\cdot,T)-u_{T}(\cdot,T)\Vert _{L^{2}(\Omega )}^{2}+\frac{N}{2}%
\Vert v\Vert _{L^{2}(\omega ^{T})}^{2}+ \\ 
+\frac{C^{n-1}\omega _{n}}{2}\Vert (u_{T}(\cdot,T)-\mathcal{B}%
_{n}H(u_0(v))(\cdot,T))\Vert _{L^{2}(\Omega )}^{2}+\frac{\mathcal{A}_{n}}{2}%
\Vert \partial _{t}H(u_0(v))\Vert _{L^{2}(Q^{T})}^{2}.%
\end{gathered}
\end{equation}%
A new consequence of the critical size of the particles is the presence of
the last two terms of $J_{0}(v)$: this is not expected if we compare what
arises in this critical case with similar results obtained for
non-critical cases. Some comments on the application of this optimal control
limit problem to the study of the approximate controllability of solutions
to the problem \eqref{limit u prob} will be given later (see Remark~\ref{Rm Controlability Limit problem}).

In order to get the proof of these convergence results, we will use the
extension of Pontryagin's method to the case of distributed problems
(following the main ideas introduced in \cite{LionsOpt}). The delicate question is to get a priori estimates under the presence of suitable
balances appearing in this critical case and to identify the limit of
several auxiliary expressions. In Section 2, we give the details of the
formulation of the direct problem as well as the coupled system arising in
terms of the adjoint optimal state $p_{\varepsilon }$, we will show that the
optimal control is given by $v_{\varepsilon }=-N^{-1}p_{\varepsilon }\chi
_{\omega _{\varepsilon }}.$ The a priori estimates allow to pass to the
limit in the couple $(u_{\varepsilon },p_{\varepsilon })$ (and thus in the
controls $v_{\varepsilon }$) is presented in Section 3. A detailed
formulation of the main theorems of this paper is collected in Section 4.
For the proofs, we start by characterizing the limit couple $(u_{0},p_{0})$
of the couple $(u_{\varepsilon },p_{\varepsilon })$, which is presented in
Section 5, and finally, we prove the identification of the limit cost functional $J_{0}(v)$ in Section 6.
 
\section{Problem formulation and coupled system with the adjoint state}
Let $\Omega$ be a bounded domain in $\mathbb{R}^n$, $n\ge 3$ with a smooth boundary $\pa\Omega$ (the case
of $n=2$ can be considered but requires a different approach: e.g. see
Section 4.7.2 of \cite{DiGoShBook} and its references). Let $Y=(-1/2, 1/2)^n$ -- a unit cube centered at the coordinate's origin, and $G_0 =\{x \vert |x| < 1 \}$. Define $\delta B = \{ x \vert \delta^{-1}x \in B \}$ for $\delta > 0$. We set, for $\ee > 0$,
\begin{equation*}
    \widetilde{\Omega}_\ee = \{ x\in\Omega \,\vert\, {\rm dist}(x, \pa\Omega) > 2\ee\}.
\end{equation*}

By $\mathbb{Z}^n$, we denote the set of all vectors $z = (z_1,\dots, z_n)$ with integer coordinates $z_i$, $i\in\overline{1, n}$. Next, we define the set
\begin{equation*}
    G_\ee = \bigcup\limits_{j\in\Upsilon_\ee}(a_\ee G_0 + \ee j) = \bigcup\limits_{j\in \Upsilon_\ee}G^j_\ee,
\end{equation*}
where $\Upsilon_\ee = \{ j\in \mathbb{Z}^n \,\vert\, \overline{G^j_\ee} \subset Y^j_\ee = \ee Y + \ee j, G^j_\ee \bigcap \widetilde{\Omega}_\ee \neq \emptyset \}$. We consider $a_\ee = C_0 \ee^{\gamma}$, $\gamma = \frac{n}{n - 2}$, $C_0$ is a positive constant. Also, note that $|\Upsilon_\ee| = d \ee^{-n}$ for some constant $d$. Next, the set $G_\ee$ is used to define the sets in \eqref{def: basic sets}.

We say that the function $u_\ee \in L^2(0, T; H^1(\Omega_\ee, \pa\Omega))$ with $\pa_t u_\ee \in L^2(0, T; L^2(\Omega_\ee))$ and  $\pa_t u_\ee \in L^2(0, T; L^2(S_\ee))$ is a weak solution to the initial boundary value problem \eqref{init state prob} if $u(x, 0) = 0$ for a.e. $x\in\Omega_\ee$ and a.e. $x\in S_\ee$, and it satisfies the integral identity
\begin{equation}\label{int ident u}
    \begin{gathered}
        \int\limits_{Q^T_\ee}\pa_t u_\ee \varphi{dxdt} + \int\limits_{Q^T_\ee}\nabla u_\ee\nabla \varphi{dxdt} + \ee^{-\gamma}\int\limits_{S^T_\ee}\pa_t u_\ee \varphi{dsdt} 
        = 
        \int\limits_{Q^T_\ee}(f + \chi_{\omega_\ee}v)\varphi{dxdt}
    \end{gathered}
\end{equation}
for an arbitrary function $\varphi \in L^2(0, T; H^1(\Omega_\ee, \pa\Omega))$.



For a given control function $v \in L^{2}(\omega _{\varepsilon }^{T})$, we will denote by $u_{\varepsilon }(v)$ a unique weak solution to the initial boundary-value problem (\ref{init
state prob}).
The existence and uniqueness of $%
u_{\varepsilon }(v)$ was shown in ~\cite{ZuShDynDok19} by means of the Galerkin
method. In addition, the following a priori estimate is valid
\begin{gather*}
\Vert \partial _{t}u_{\varepsilon }(v)\Vert _{L^{2}(Q_{\varepsilon
}^{T})}^{2}+\varepsilon ^{-\gamma }\Vert \partial _{t}u_{\varepsilon
}(v)\Vert _{L^{2}(S_{\varepsilon }^{T})}^{2}+\max\limits_{t\in \lbrack
0,T]}\Vert \nabla u_{\varepsilon }(v)(\cdot ,t)\Vert _{L^{2}(\Omega
_{\varepsilon })}^{2} \leq \\
\leq K(\Vert f\Vert _{L^{2}(Q^{T})}^{2}+\Vert v\Vert _{L^{2}(\omega
_{\varepsilon }^{T})}^{2}),
\end{gather*}
here and below, we suppose that $K$ is independent of $\varepsilon$.

\begin{remark}
\label{rm Abstract theory}The existence and uniqueness of solutions to the
problem (\ref{init state prob}) can be obtained via the abstract theory of
subdifferential of convex functions. Indeed, by an easy adaptation of the
proof of Theorem 1.3 of \cite{BejeDVrabie}, we know that the vectorial
operator associated to this system is a self-adjoint operator $A_{H}$ on the
Hilbert space $H=L^{2}(\Omega _{\varepsilon })\times L^{2}(S_{\varepsilon })$
and thus, by Proposition 2.15 of \cite{Brezis}, it is the subdifferential of
a convex function defined through the square root of $A_{H}.$ Finally, to
take into account the coefficient $\varepsilon ^{-\gamma }$ in the dynamic
boundary condition, we apply Propositions 3.2 and 3.3 of \cite{Showalter}
with $\mathcal{B}$ the matrix%
\begin{equation*}
\mathcal{B=}\left( 
\begin{array}{cc}
I & 0 \\ 
0 & \varepsilon ^{-\gamma }I%
\end{array}%
\right) .
\end{equation*}
\end{remark}

As mentioned in the Introduction, along with the problem (\ref{init state prob}%
), we consider the cost functional $J_{\varepsilon }:L^{2}(\omega
_{\varepsilon }^{T})\rightarrow \mathbb{R}$ that is defined in (\ref{cost
functional}) for a given $u_{T}\in H^{1}(0,T;H_{0}^{1}(\Omega)) \bigcap C(\overline{Q^T})$. By
well-known results \cite{LionsOpt}, \cite{Fursikov2000OptimalCO}, \cite%
{Trolstz}, \cite{Glow-Lions}, there exists a unique optimal control $%
v_{\varepsilon }\in L^{2}(\omega _{\varepsilon }^{T})$ such that $%
J_{\varepsilon }(v_{\varepsilon })=\inf\limits_{v\in L^{2}(\omega
_{\varepsilon }^{T})}J_{\varepsilon }(v).$

We define the adjoin problem associated to the given state problem \eqref{init state prob} and the cost functional \eqref{cost functional}
\begin{equation}\label{init adjoint prob}
\left\{
    \begin{array}{lr}
        \pa_t p_\ee + \Delta p_\ee = \Delta (u_\ee - u_T), & (x,t)\in Q^T_\ee,\\
        \pa_\nu p_\ee -\ee^{-\gamma}\pa_t p_\ee = \pa_\nu (u_\ee - u_T), & (x, t)\in S^T_\ee,\\
        p_\ee(x, T) = (u_\ee - u_T)(x, T), & x\in \Omega_\ee,\\
        p_\ee(x, T) = (u_\ee - u_T)(x, T), & x\in S_\ee,\\
        p_\ee(x, t) = 0, & (x, t) \in \Gamma^T,
    \end{array}
\right.
\end{equation}
We say that the function $%
p_{\varepsilon }\in L^{2}(0,T; H^1(\Omega_\ee, \pa\Omega))$ 
with $\partial
_{t}p_{\varepsilon }\in L^{2}(0,T;L^{2}(\Omega _{\varepsilon }))$ and $%
\partial _{t}p_{\varepsilon }\in L^{2}(0,T;L^{2}(S_{\varepsilon }))$ is a weak solution to the problem \eqref{init adjoint prob} if $p_\ee(x , T) = (u_\ee - u_T)(x, T)$ for a.e. $x\in \Omega_\ee$ and a.e. $x\in S_\ee$ and if it satisfies the integral identity
\begin{equation}\label{int ident init adjoint prob}
    -\int\limits_{Q^T_\ee}\pa_t p_\ee \varphi{dxdt} + \int\limits_{Q^T_\ee}\nabla p_\ee \nabla\varphi{dxdt} - \ee^{-\gamma}\int\limits_{S^T_\ee}\pa_t p_\ee \varphi {dsdt} = \int\limits_{Q^T_\ee}\nabla (u_\ee - u_T)\nabla\varphi{dxdt},
\end{equation}
for any test function $\varphi\in L^2(0, T; H^1(\Omega_\ee, \pa\Omega))$. The following results are a version of the \textit{Pontryagin's
principle} for this optimality system.

\begin{theorem}\label{thm: init opt cont}
    Let the pair of function $(u_\ee(v_\ee), v_\ee)$ be an optimal solution of the problem \eqref{optim control def}, then $v_\ee = -N^{-1}p_\ee\chi_{\omega_\ee}$, where $p_\ee$ is the solution to \eqref{init adjoint prob}. The converse is also true.
\end{theorem}
\begin{proof}
    Let $v$ be an arbitrary function from $L^2(\omega^T_\ee)$ and $\lambda > 0$. By $u^\lambda_\ee$, we denote the solution of \eqref{optim control def} with the control $v_\ee + \lambda v$, i.e. $u^\lambda_\ee = u_\ee(v_\ee + \lambda v)$. We use $u_\ee = u_\ee(v_\ee)$ to simplify the notation.

    Indeed, we have
    \begin{gather*}
        J_\ee(v_\ee + \lambda v) - J_\ee(v_\ee) = \frac{1}{2}\Vert\nabla (u^\lambda_\ee - u_T)\Vert^2_{L^2(Q^T_\ee)} + \frac{1}{2}\Vert (u^\lambda_\ee - u_T)(x, T)\Vert^2_{L^2(\Omega_\ee)} +\\
        + \frac{\ee^{-\gamma}}{2}\Vert (u^\lambda_\ee - u_T)(x, T) \Vert^2_{L^2(S_\ee)} 
        + \frac{N}{2}\Vert v_\ee + \lambda v\Vert^2_{L^2(\omega^T_\ee)}-\\
        - \frac{1}{2}\Vert\nabla (u_\ee - u_T)\Vert^2_{L^2(Q^T_\ee)} - \frac{1}{2}\Vert (u_\ee - u_T)(x, T)\Vert^2_{L^2(\Omega_\ee)} -\\
        - \frac{\ee^{-\gamma}}{2}\Vert (u_\ee - u_T)(x, T) \Vert^2_{L^2(S_\ee)} - \frac{N}{2}\Vert v_\ee\Vert^2_{L^2(\omega^T_\ee)} = \\
        = \frac{1}{2}\int\limits_{Q^T_\ee}\nabla(u^\lambda_\ee - u_\ee)\nabla(u^\lambda_\ee + u_\ee - 2 u_T) {dxdt} + \frac{1}{2}\int\limits_{\Omega_\ee}(u^\lambda_\ee - u_\ee)(u^\lambda_\ee + u_\ee - 2 u_T)(x, T){dx} + \\
        + \frac{\ee^{-\gamma}}{2}\int\limits_{S_\ee}(u^\lambda_\ee - u_\ee)(u^\lambda_\ee + u_\ee - 2 u_T)(x, T){ds} + \frac{N}{2}\int\limits_{\omega^T_\ee}(2\lambda v_\ee v + \lambda^2 v^2){dxdt}
    \end{gather*}

Next, we define the function $\theta^\lambda_\ee = u^\lambda_\ee - u_\ee$. It is easy to see that this function is a solution to the problem
\begin{equation*}
\left\{
    \begin{array}{lr}
        \pa_t \theta^\lambda_\ee - \Delta \theta^\lambda_\ee = \lambda v \chi_{w_\ee}, & (x,t)\in Q^T_\ee,\\
        \ee^{-\gamma}\pa_t \theta^\lambda_\ee + \pa_\nu \theta^\lambda_\ee = 0, & (x, t)\in S^T_\ee,\\
        \theta^\lambda_\ee(x, 0) = 0, & x\in \Omega_\ee,\\
        \theta^\lambda_\ee(x, 0) = 0, & x\in S_\ee,\\
        \theta^\lambda_\ee(x, t) = 0, & (x, t) \in \Gamma^T,
    \end{array}
\right.
\end{equation*}
The solution to this problem satisfies the estimate
\begin{gather*}
    \Vert \pa_t \theta^\lambda_\ee \Vert_{L^2(Q^T_\ee)} + \ee^{-\gamma}\Vert \pa_t \theta^\lambda_\ee \Vert_{L^2(S^T_\ee)} +\\
    + \max\limits_{t\in [0, T]}\Vert \nabla\theta^\lambda_\ee(\cdot, t) \Vert_{L^2(\Omega_\ee)} + \ee^{-\gamma}\max\limits_{t\in[0, T]}\Vert \theta^\lambda_\ee(\cdot, t)\Vert_{L^2(S_\ee)}  \le K|\lambda| \Vert v\Vert_{L^2(\omega^T_\ee)}.
\end{gather*}
Thus, we have the following convergences as $\lambda\to 0$
\begin{gather*}
    u^\lambda_\ee \to u_\ee \mbox{ in } L^2(0, T; H^1(\Omega_\ee, \pa\Omega)), \quad\pa_t u^\lambda_\ee \to \pa_t u_\ee \mbox{ in } L^2(0, T; L^2(\Omega_\ee)),\\
    \pa_t u^\lambda_\ee \to \pa_t u_\ee \mbox{ in } L^2(0, T; L^2(S_\ee)),\quad
    u^\lambda_\ee(\cdot, T) \to u_\ee(\cdot, T) \mbox{ in } L^2(\Omega_\ee),\\
    u^\lambda_\ee(\cdot, T) \to u_\ee(\cdot, T) \mbox{ in } L^2(S_\ee).
\end{gather*}
We define the function $\theta_\ee = \theta^\lambda_\ee / \lambda$ that is a solution to the following problem
\begin{equation*}
\left\{
    \begin{array}{lr}
        \pa_t \theta_\ee - \Delta \theta_\ee = v\chi_{\omega_\ee}, & (x,t)\in Q^T_\ee,\\
        \ee^{-\gamma}\pa_t \theta_\ee + \pa_\nu \theta_\ee = 0, & (x, t)\in S^T_\ee,\\
        \theta_\ee(x, 0) = 0, & x\in \Omega_\ee,\\
        \theta_\ee(x, 0) = 0, & x\in S_\ee,\\
        \theta_\ee(x, t) = 0, & (x, t) \in \Gamma^T.
    \end{array}
\right.
\end{equation*}
Now, as usual, we divide $J_\ee(v_\ee + \lambda v) - J_\ee(v_\ee)$ by $\lambda$ and pass to the limit as $\lambda\to 0$. Thus, we get
\begin{gather*}
    J'_\ee(v_\ee)v = \lim\limits_{\lambda\to 0}(J_\ee(v_\ee + \lambda v) - J_\ee(v_\ee))/\lambda =\\
    = \int\limits_{Q^T_\ee}\nabla \theta_\ee \nabla (u_\ee - u_T) {dxdt} + \int\limits_{\Omega_\ee}\theta_\ee(x, T) (u_\ee - u_T)(x, T){dx} + \\
    +\ee^{-\gamma}\int\limits_{S_\ee}\theta_\ee(x, T) (u_\ee - u_T)(x, T){ds} + N\int\limits_{\omega^T_\ee}v_\ee v{dxdt}
\end{gather*}
Using the definition of the function $p_\ee$, we transform the last expression into
\begin{gather*}
    J'_\ee(v_\ee)v = \int\limits_{\omega^T_\ee}p_\ee v{dxdt} + N\int\limits_{\omega^T_\ee}v_\ee v{dxdt}.
\end{gather*}
The function $v_\ee$ is the optimal control, hence, $J'_\ee(v_\ee)v = 0$ for an arbitrary function $v\in L^2(\omega^T_\ee)$. This implies that $v_\ee = -N^{-1}p_\ee\chi_{\omega_\ee}$.
    
\end{proof}

According to the above theorem, the optimality conditions for $%
(u_{\varepsilon },p_{\varepsilon })$ lead to the study of the coupled system
\begin{equation}\label{init coupled prob}
\left\{
    \begin{array}{lr}
        \pa_t u_\ee - \Delta u_\ee = f - N^{-1}p_\ee\chi_{\omega_\ee}, & (x, t)\in Q^T_\ee,\\
        \pa_t p_\ee + \Delta p_\ee = \Delta (u_\ee - u_T), & (x,t)\in Q^T_\ee,\\
        \pa_\nu u_\ee + \ee^{-\gamma}\pa_t u_\ee = 0, & (x, t)\in S^T_\ee,\\
        \pa_\nu p_\ee - \ee^{-\gamma}\pa_t p_\ee = \pa_\nu (u_\ee - u_T), & (x, t)\in S^T_\ee,\\
        u_\ee(x, 0) = 0, & x\in \Omega_\ee,\\
        u_\ee(x, 0) = 0, & x\in S_\ee,\\
        p_\ee(x, T) = (u_\ee - u_T)(x, T), & x\in \Omega_\ee,\\
        p_\ee(x, T) = (u_\ee - u_T)(x, T), & x\in S_\ee,\\
        u_\ee(x, t) = p_\ee(x, t) = 0, & (x, t) \in \Gamma^T.
    \end{array}
\right.
\end{equation}

\begin{remark}
\label{Rm Controlability probl init} By following some arguments introduced
by J.-L. Lions in \cite{Lions Malaga} (see also, e.g., Section 1.6 in the
book \cite{Glow-Lions}), it is possible to use the study of a small
adaptation of the above optimal control problem to prove the
\textquotedblleft approximate controllability property\textquotedblright\
for solutions of problem (\ref{init state prob}), provided we know a result
on \textquotedblleft unique continuation\textquotedblright\ for this
problem. Given $u_{T}\in H^{1}(0,T;H_{0}^{1}(\Omega ))\bigcap C(\overline{Q^T})$ and an arbitrarily
small $\delta >0,$ the \textquotedblleft approximate controllability
property\textquotedblright\ consists in finding a control $v_{\delta }\in
L^{2}(\omega _{\varepsilon }^{T})$ such that $\left\Vert u_{\varepsilon
}(\cdot,T)-u_{T}(\cdot,T)\right\Vert _{L^{2}(\Omega _{\varepsilon })}\leq \delta $
and $\left\Vert u_{\varepsilon }(\cdot,T)-u_{T}(\cdot,T)\right\Vert
_{L^{2}(S_{\varepsilon })}\leq \delta $. By using some a priori estimates on
the adjoint state $p_{\varepsilon }(x,t)$, it can be shown that if we
introduce a new parameter $\kappa >0$ in the cost functional%
\begin{equation*}
\begin{gathered}
J_{\varepsilon }^{\kappa }(v)=\frac{\kappa }{2}\int\limits_{Q_{\varepsilon
}^{T}}|\nabla u_{\varepsilon }(v)-\nabla u_{T}|^{2}{dxdt}+\frac{\kappa }{2}%
\int\limits_{\Omega _{\varepsilon }}(u_{\varepsilon }(v)-u_{T})^{2}(x,T){dx}
\\ 
+\frac{\kappa {\varepsilon }^{-\gamma }}{2}\int\limits_{S_{\varepsilon
}}(u_{\varepsilon }(v)-u_{T})^{2}(x,T){ds}+\frac{N}{2}\int\limits_{\omega
_{\varepsilon }^{T}}v^{2}{dxdt},%
\end{gathered}%
\end{equation*}%
then a such searched control $v_{\delta }$ can be found by considering the
set of optimal controls $v_{\kappa }\in L^{2}(\omega _{\varepsilon }^{T})$
associated to $J_{\varepsilon }^{\kappa }(v)$ and by taking $v_{\delta
}=v_{\kappa }$ for $\kappa $ large enough. As a matter of facts, the
presence of the gradient difference term (the first term) in $J_{\varepsilon
}^{\kappa }(v)$ leads to conclude that the associated state will also
satisfy the additional property $\left\Vert u_{\varepsilon
}-u_{T}\right\Vert _{L^{2}(Q_{\varepsilon }^{T})}\leq \delta $.
\end{remark}

\section{A priori estimates on $u_{\protect\varepsilon }$ and $p_{\protect\varepsilon }$}

Taking $p_\ee$ as a test function in the integral identity for $u_\ee$, we get
\begin{equation}\label{sec3: int ident u p}
    \begin{gathered}
        \int\limits_{Q^T_\ee}\pa_t u_\ee p_\ee {dxdt} + \ee^{-\gamma}\int\limits_{S^T_\ee}\pa_t u_\ee p_\ee{dsdt} + \int\limits_{Q^T_\ee}\nabla u_\ee \nabla p_\ee{dxdt} 
        = \int\limits_{Q^T_\ee}f p_\ee{dxdt} - \frac{1}{N}\int\limits_{\omega^T_\ee}p^2_\ee{dxdt}.
    \end{gathered}
\end{equation}
Now, taking $u_\ee$ as a test function in the integral identity for $p_\ee$, we get
\begin{equation}\label{sec3: int ident p u}
    \begin{gathered}
        -\int\limits_{Q^T_\ee}\pa_t p_\ee u_\ee {dxdt} - \ee^{-\gamma}\int\limits_{S^T_\ee}\pa_t p_\ee u_\ee{dsdt} + \int\limits_{Q^T_\ee}\nabla p_\ee \nabla u_\ee{dxdt} 
        = \int\limits_{Q^T_\ee}|\nabla u_\ee|^2{dxdt} - \int\limits_{Q^T_\ee}\nabla u_\ee\nabla u_T{dxdt}.
    \end{gathered}
\end{equation}
Next, we subtract \eqref{sec3: int ident u p} from \eqref{sec3: int ident p u} and obtain the expression
\begin{equation*}
    \begin{gathered}
        -\int\limits_{Q^T_\ee}\pa_t(u_\ee p_\ee){dxdt} - \ee^{-\gamma}\int\limits_{S^T_\ee}\pa_t(u_\ee p_\ee){dsdt} = \\
        = \int\limits_{Q^T_\ee}|\nabla u_\ee|^2{dxdt} - \int\limits_{Q^T_\ee}f p_\ee{dxdt} - \int\limits_{Q^T_\ee}\nabla u_\ee \nabla u_T{dxdt} + \frac{1}{N}\int\limits_{\omega^T_\ee}p_\ee^2{dxdt}.
    \end{gathered}
\end{equation*}
Transforming this equality, we get
\begin{gather*}
    \Vert\nabla u_\ee\Vert^2_{L^2(Q^T_\ee)} + \Vert u_\ee(\cdot, T)\Vert^2_{L^2(\Omega_\ee)} + \ee^{-\gamma}\Vert u_\ee(\cdot, T)\Vert^2_{L^2(S_\ee)} + \frac{1}{N}\Vert p_\ee\Vert^2_{L^2(\omega^T_\ee)} = \\
    = \int\limits_{Q^T_\ee}f p_\ee {dxdt} + \int\limits_{Q^T_\ee}\nabla u_\ee \nabla u_T{dxdt} + \int\limits_{\Omega_\ee}u_\ee(x, T)u_T(x, T){dx} + \ee^{-\gamma}\int\limits_{S_\ee}u_\ee(x, T)u_T(x, T){ds}.
\end{gather*}
From here, we immediately derive
\begin{gather*}
    \Vert\nabla u_\ee\Vert^2_{L^2(Q^T_\ee)} + \Vert u_\ee(\cdot, T)\Vert^2_{L^2(\Omega_\ee)} + \ee^{-\gamma}\Vert u_\ee(\cdot, T)\Vert^2_{L^2(S_\ee)} + \Vert p_\ee\Vert^2_{L^2(\omega^T_\ee)} \le \\ 
    \le K(\Vert f\Vert_{L^2(Q^T)}\Vert p_\ee\Vert_{L^2(Q^T_\ee)} + \Vert\nabla u_T\Vert^2_{L^2(Q^T_\ee)} + \max\limits_{x\in\overline\Omega}|u_T(x, T)|^2).
\end{gather*}
Then, we take $p_\ee$ as a test function in the integral identity for $p_\ee$ and get
\begin{gather*}
    \Vert \nabla p_\ee \Vert^2_{L^2(Q^T_\ee)} \le K(\Vert\nabla (u_\ee - u_T)\Vert^2_{L^2(Q^T_\ee)} + \Vert (u_\ee - u_T)(\cdot, T)\Vert^2_{L^2(\Omega_\ee)} +\\
    +\ee^{-\gamma}\Vert (u_\ee - u_T)(\cdot, T)\Vert^2_{L^2(S_\ee)} )
    \le K(\Vert f\Vert_{L^2(Q^T)}\Vert p_\ee\Vert_{L^2(Q^T_\ee)} + \Vert\nabla u_T\Vert^2_{L^2(Q^T_\ee)} + \max\limits_{x\in\overline\Omega}|u_T(x, T)|^2).
\end{gather*}
Now, for the function from $H^1(\Omega_\ee, \pa\Omega)$, we have Friedrichs's inequality
\begin{equation*}
    \Vert p_\ee\Vert_{L^2(\Omega_\ee)} \le K\Vert \nabla p_\ee\Vert_{L^2(\Omega_\ee)}.
\end{equation*}
Applying it to the previous inequality, we derive
\begin{equation*}
    \Vert p_\ee\Vert^2_{L^2(Q^T_\ee)} \le K (\Vert f\Vert_{L^2(Q^T)}\Vert p_\ee\Vert_{L^2(Q^T_\ee)} + \Vert\nabla u_T\Vert^2_{L^2(Q^T_\ee)} + \max\limits_{x\in\overline\Omega}|u(x, T)|^2).
\end{equation*}
Hence, we obtain the estimation
\begin{equation*}
    \Vert p_\ee \Vert^2_{L^2(Q^T_\ee)} \le K (\Vert f\Vert^2_{L^2(Q^T)} + \Vert\nabla u_T\Vert^2_{L^2(Q^T_\ee)} + \max\limits_{x\in\overline\Omega}|u_T(x, T)|^2).
\end{equation*}
Then, we immediately derive the estimations
\begin{equation}\label{u p estimations}
\begin{gathered}
    \Vert\nabla u_\ee\Vert^2_{L^2(Q^T_\ee)} + \Vert u_\ee(\cdot, T)\Vert^2_{L^2(\Omega_\ee)} + \ee^{-\gamma}\Vert u_\ee(\cdot, T)\Vert^2_{L^2(S_\ee)} + \Vert p_\ee\Vert^2_{L^2(Q^T_\ee)} \le\\
    \le K(\Vert f\Vert^2_{L^2(Q^T)} + \Vert \nabla u_T\Vert^2_{L^2(Q^T)} + \max\limits_{x\in\overline\Omega}|u_T(x, T)|^2),\\
    \Vert \nabla p_\ee \Vert^2_{L^2(Q^T_\ee)} \le K(\Vert f\Vert^2_{L^2(Q^T)} + \Vert\nabla u_T\Vert^2_{L^2(Q^T)} + \max\limits_{x\in\overline\Omega}|u_T(x, T)|^2).
\end{gathered}
\end{equation}

Next, we get estimations for the time derivatives of $u_\ee$ and $p_\ee$. Following the work \cite{ZuShDynDok19}, we construct Galerkin's approximations of $u_\ee$ and $p_\ee$ which we denote by $u^m_\ee$ and $p^m_\ee$ respectively, where $m=1,2, \dots$. Note that for $u^m_\ee$ and $p^m_\ee$, we have the same estimates as in \eqref{u p estimations} for $u_\ee$ and $p_\ee$.

Taking $\pa_t u^m_\ee$ as a test function in the equation for $u^m_\ee$ and integrating from 0 to an arbitrary $\theta\in [0, T]$, we have
\begin{gather*}
    \Vert\pa_t u^m_\ee\Vert^2_{L^2(Q^T_\ee)} + \ee^{-\gamma}\Vert\pa_t u^m_\ee\Vert^2_{L^2(S^T_\ee)} + \max\limits_{t\in[0, T]}\Vert\nabla u^m_\ee\Vert^2_{L^2(\Omega_\ee)} \le\\ \le K \int\limits_{Q^T_\ee}(|f| + \frac{1}{N}|p^m_\ee|\chi_{\omega_\ee})|\pa_t u^m_\ee|{dxdt}
    \le \frac{1}{2}\Vert \pa_t u^m_\ee\Vert^2_{L^2(Q^T_\ee)} + K(\Vert f\Vert^2_{L^2(Q^T)} + \Vert p^m_\ee\Vert^2_{L^2(w^T_\ee)}),
\end{gather*}
where constant $K$ is independent of $\ee$ and $m$. From here, we immediately derive
\begin{equation*}
    \Vert\pa_t u^m_\ee\Vert^2_{L^2(Q^T_\ee)} + \ee^{-\gamma}\Vert \pa_t u^m_\ee\Vert^2_{L^2(S^T_\ee)} 
    \le K(\Vert f\Vert^2_{L^2(Q^T)} + \Vert \nabla u_T\Vert^2_{L^2(Q^T)} + \max\limits_{x\in\overline\Omega}|u_T(x, T)|^2).
\end{equation*}
Passing to the limit as $m\to\infty$, we get the estimation
\begin{equation*}
    \Vert\pa_t u_\ee\Vert^2_{L^2(Q^T_\ee)} + \ee^{-\gamma}\Vert \pa_t u_\ee\Vert^2_{L^2(S^T_\ee)} 
    \le K(\Vert f\Vert^2_{L^2(Q^T)} + \Vert \nabla u_T\Vert^2_{L^2(Q^T)} + \max\limits_{x\in\overline\Omega}|u_T(x, T)|^2).
\end{equation*}

Again using Galerkin's approximation, we have for a.e. $t\in (0, T)$
\begin{gather*}
    -\Vert\pa_t p^m_\ee\Vert^2_{L^2(\Omega_\ee)} - \ee^{-\gamma}\Vert\pa_t p^m_\ee\Vert^2_{L^2(S_\ee)} + (\nabla p^m_\ee, \pa_t \nabla p^m_\ee)_{L^2(\Omega_\ee)}=\\
    = -(\pa_t u^m_\ee, \pa_t p^m_\ee)_{L^2(\Omega_\ee)} - \ee^{-\gamma}(\pa_t u^m_\ee, \pa_t p^m_\ee)_{L^2(S_\ee)} +\\
    + ((f - N^{-1}p^m_\ee\chi_{\omega_\ee}), \pa_t p^m_\ee)_{L^2(\Omega_\ee)} - (\nabla u_T, \nabla\pa_t p^m_\ee)_{L^2(\Omega_\ee)}.
\end{gather*}

Integrating this expression with respect to $t$ from 0 to $T$, 
we get
\begin{gather*}
    -\Vert \pa_t p^m_\ee\Vert^2_{L^2(Q^T_\ee)} + \frac{1}{2}\Vert\nabla p^m_\ee(\cdot, T)\Vert^2_{L^2(\Omega_\ee)} - \frac{1}{2}\Vert\nabla p^m_\ee(\cdot, 0)\Vert^2_{L^2(\Omega_\ee)} - \ee^{-\gamma}\Vert \pa_t p^m_\ee\Vert^2_{L^2(S^T_\ee)} = \\
    = -\int\limits_{Q^T_\ee}\pa_t u^m_\ee \pa_t p^m_\ee{dxdt} - \ee^{-\gamma}\int\limits_{S^T_\ee}\pa_t u^m_\ee \pa_t p^m_\ee {dsdt} + \\
    + \int\limits_{Q^T_\ee}(f - N^{-1}p^m_\ee\chi_{\omega_\ee})\pa_t p^m_\ee{dxdt} - \int\limits_{Q^T_\ee}\nabla u_T \nabla\pa_t p^m_\ee{dxdt} =\\
    = -\int\limits_{Q^T_\ee}\pa_t u^m_\ee \pa_t p^m_\ee{dxdt} - \ee^{-\gamma}\int\limits_{S^T_\ee}\pa_t u^m_\ee \pa_t p^m_\ee {dsdt} + \\
    + \int\limits_{Q^T_\ee}(f - N^{-1}p^m_\ee\chi_{\omega_\ee})\pa_t p^m_\ee{dxdt} +   \int\limits_{Q^T_\ee}\nabla \pa_t u_T \nabla p^m_\ee{dxdt} - \\
    - \int\limits_{\Omega_\ee}\nabla u_T(x, T)\nabla p^m_\ee(x, T){dx} + \int\limits_{\Omega_\ee}\nabla u_T(x, 0)\nabla p^m_\ee(x, 0){dx}.
\end{gather*}
From here, we derive the estimate
\begin{gather*}
    \Vert \pa_t p^m_\ee\Vert^2_{L^2(Q^T_\ee)} + \ee^{-\gamma}\Vert \pa_t p^m_\ee\Vert^2_{L^2(S^T_\ee)} + \frac{1}{2}\Vert \nabla p^m_\ee(\cdot, 0)\Vert^2_{L^2(\Omega_\ee)} + \frac{1}{2N}\Vert p^m_\ee(x, 0)\Vert^2_{L^2(\omega_\ee)} \le\\
    \le\frac{1}{2}\Vert\nabla (u^m_\ee - u_T)(\cdot, T)\Vert^2_{L^2(\Omega_\ee)} + \Vert \pa_t u^m_\ee\Vert_{L^2(Q^T_\ee)} \Vert\pa_t p^m_\ee\Vert_{L^2(Q^T_\ee)} + \\
    + \ee^{-\gamma}\Vert \pa_t u^m_\ee\Vert_{L^2(S^T_\ee)} \Vert\pa_t p^m_\ee\Vert_{L^2(S^T_\ee)} + \Vert f\Vert_{L^2(Q^T_\ee)}\Vert\pa_t p^m_\ee\Vert_{L^2(Q^T_\ee)} + \\
    + \frac{1}{2N}\Vert (u^m_\ee - u_T)(\cdot, T)\Vert^2_{L^2(\omega_\ee)}
    + \Vert\nabla\pa_t u_T\Vert_{L^2(Q^T_\ee)}\Vert\nabla p^m_\ee\Vert_{L^2(Q^T_\ee)}+
    \\ + \Vert\nabla u_T(\cdot, T)\Vert_{L^2(\Omega_\ee)}\Vert \nabla (u^m_\ee - u_T)(\cdot, T)\Vert_{L^2(\Omega_\ee)} + \Vert\nabla u_T(\cdot, 0)\Vert_{L^2(\Omega_\ee)}\Vert \nabla p^m_\ee(\cdot, 0)\Vert_{L^2(\Omega_\ee)}.
\end{gather*}
Next, we get
\begin{gather*}
    \Vert \pa_t p^m_\ee\Vert^2_{L^2(Q^T_\ee)} + \ee^{-\gamma}\Vert \pa_t p^m_\ee\Vert^2_{L^2(S^T_\ee)} + \frac{1}{4}\Vert \nabla p^m_\ee(\cdot, 0)\Vert^2_{L^2(\Omega_\ee)} + \frac{1}{2N}\Vert p^m_\ee(x, 0)\Vert^2_{L^2(\omega_\ee)} \le \\ 
    \le \frac{1}{2}\Vert\nabla u^m_\ee(\cdot, T)\Vert^2_{L^2(\Omega_\ee)} + \frac{1}{4} \Vert\pa_t p^m_\ee\Vert^2_{L^2(Q^T_\ee)} + \frac{\ee^{-\gamma}}{4}\Vert\pa_t p^m_\ee\Vert^2_{L^2(S^T_\ee)} +\\
    + K(\Vert \pa_t u^m_\ee\Vert^2_{L^2(Q^T_\ee)} + \ee^{-\gamma}\Vert \pa_t u^m_\ee\Vert^2_{L^2(S^T_\ee)} + \Vert u^m_\ee(\cdot, T)\Vert^2_{L^2(\omega_\ee)} + \\
    + \Vert f\Vert^2_{L^2(Q^T)} + \Vert \nabla \pa_t u_T\Vert^2_{L^2(Q^T)} + \max\limits_{t\in[0,T]}\Vert\nabla u_T\Vert^2_{L^2(\Omega)} + \max\limits_{x\in\overline\Omega}|u_T(x, T)|^2).
\end{gather*}
Eventually, using the estimates for $u^m_\ee$, we obtain
\begin{equation}
\begin{gathered}
    \Vert \pa_t p^m_\ee\Vert^2_{L^2(Q^T_\ee)} + \ee^{-\gamma}\Vert \pa_t p^m_\ee\Vert^2_{L^2(S^T_\ee)}
    \le K(\Vert f\Vert^2_{L^2(Q^T)} + \Vert \nabla u_T\Vert^2_{L^2(Q^T)} + \\ + \Vert \nabla \pa_t u_T\Vert^2_{L^2(Q^T)} + \max\limits_{t\in[0,T]}\Vert\nabla u_T\Vert^2_{L^2(\Omega)} + \max\limits_{x\in\overline\Omega}|u_T(x, T)|^2),
\end{gathered}
\end{equation}
where constant $K$ is independent of $m$ and $\ee$. Passing to the limit as $m\to \infty$, we get the estimate for $\pa_t p_\ee$
\begin{equation}
\begin{gathered}
    \Vert \pa_t p_\ee\Vert^2_{L^2(Q^T_\ee)} + \ee^{-\gamma}\Vert \pa_t p_\ee\Vert^2_{L^2(S^T_\ee)}\le \\ 
    \le K(\Vert f\Vert^2_{L^2(Q^T)} + \Vert \nabla \pa_t u_T\Vert^2_{L^2(Q^T)} + \max\limits_{t\in[0,T]}\Vert\nabla u_T(\cdot, t)\Vert^2_{L^2(\Omega)} + \max\limits_{x\in\overline\Omega}|u_T(x, T)|^2).
\end{gathered}
\end{equation}
Thus, we got uniform in $\ee$ estimates for $u_\ee$ and $p_\ee$, and their time derivatives.

There exists the extension operator $P_\ee: H^1(Q^T_\ee) \to H^1(Q^T)$ (see. \cite{Cior99}, \cite{OlSh95}) such that 
\begin{equation*}
    \Vert P_\ee(u)\Vert_{H^1(Q^T)} \le \Vert u\Vert_{H^1(Q^T_\ee)}.
\end{equation*}
Let $\tilde u_\ee$, $\tilde p_\ee$ be the extensions of the functions $u_\ee$, $p_\ee$ then the following estimations are valid
\begin{gather}
    \Vert \pa_t \tilde u_\ee\Vert^2_{L^2(Q^T)} + \Vert \nabla \tilde u_\ee\Vert^2_{L^2(Q^T)} \le K(\Vert \pa_t  u_\ee\Vert^2_{L^2(Q^T_\ee)} + \Vert \nabla u_\ee\Vert^2_{L^2(Q^T_\ee)}),\label{ext u estim}\\
    \Vert \pa_t \tilde p_\ee\Vert^2_{L^2(Q^T)} + \Vert \nabla \tilde p_\ee\Vert^2_{L^2(Q^T)} \le K(\Vert \pa_t  p_\ee\Vert^2_{L^2(Q^T_\ee)} + \Vert \nabla p_\ee\Vert^2_{L^2(Q^T_\ee)}).\label{ext p estim}
\end{gather}

The estimations \eqref{ext u estim}, \eqref{ext p estim} imply that there exist subsequences for which we preserve the notation of the original, i.e. $\tilde u_\ee$ and $\tilde p_\ee$, such that, as $\ee\to0$, we have
\begin{equation}\label{limit func def}
    \begin{gathered}
        \tilde u_\ee \rightharpoonup u_0 \mbox{ weakly in } L^2(0, T; H^1_0(\Omega)),\\
        \tilde p_\ee \rightharpoonup p_0 \mbox{ weakly in } L^2(0, T; H^1_0(\Omega)),\\
        \pa_t \tilde u_\ee \rightharpoonup \pa_t u_0 \mbox{ weakly in } L^2(Q^T_\ee),\\
        \pa_t \tilde p_\ee \rightharpoonup \pa_t p_0 \mbox{ weakly in } L^2(Q^T_\ee).
    \end{gathered}
\end{equation}
Moreover, the embedding theorem implies that $\tilde u_\ee \to u_0$ and $\tilde p_\ee \to p_0$ in $L^2(Q^T)$. 

\section{Main results}
The following theorem characterizes the limit functions $u_0$ and $p_0$ defined in \eqref{limit func def}.

\begin{theorem}\label{thm: main theorem}
    Let $n\ge 3$, $a_\ee = C_0 \ee^{\gamma}$, where $C_0 > 0$, $\gamma = n / (n-2)$. If the pair $(u_\ee, p_\ee)$ is a solution to the problem \eqref{init coupled prob}, then, the pair $(u_0, p_0)$ is a solution to the system
    \begin{equation}\label{limit coupled prob}
    \left\{
        \begin{array}{lr}
            \pa_t u_0 - \Delta u_0 + \mathcal{A}_n(u_0 - \mathcal{B}_n H(u_0)) = f - N^{-1}\chi_{\omega}p_0, & (x, t)\in Q^T,\\
            -\pa_t p_0 - \Delta p_0 + \mathcal{A}_n(p_0 - \mathcal{B}_n H^*(p_0)) =\\ -\Delta (u_0 - u_T) + \mathcal{A}_n (u_0 - \mathcal{B}_n^2 H^*(H(u_0)) - e^{\mathcal{B}_n(t - T)}u_T(x, T)), & (x,t)\in Q^T,\\
            u_0(x, 0) = 0, & x\in \Omega,\\
            p_0(x, T) = (u_0 - u_T)(x, T), &  x\in \Omega,\\
            u_0(x, t) = p_0(x, t) = 0, & (x, t) \in \Gamma^T,
        \end{array}
    \right.
    \end{equation}
    where $Q^T = \Omega \times (0, T)$, $\mathcal{A}_n = (n-2) C^{n-2}_0 w_n$, $H(\varphi)(x, t)$ is given as a solution to
    \begin{equation}\label{H prob}
        \left\{
        \begin{array}{lr}
            \pa_t H(\varphi) + \mathcal{B}_n H(\varphi) = \varphi, & t\in (0, T),\\
            H(\varphi)(x, 0) = 0,
        \end{array}
        \right.
    \end{equation}
    where $\mathcal{B}_n = (n - 2)C^{-1}_0$, and $H^*(\varphi)(x, t)$ is a solution to the adjoint problem
    \begin{equation}\label{adj H prob}
        \left\{
            \begin{array}{lr}
            -\pa_t H^*(\varphi) + \mathcal{B}_n H^*(\varphi) = \varphi, & t\in (0, T),\\
            H^*(\varphi)(x, T) = 0.
        \end{array}
        \right.
    \end{equation}
    
\end{theorem}
\begin{remark}
    In the problems \eqref{H prob}, \eqref{adj H prob}, we can view $x$ as a parameter, and for a.e. $x$, we have
    \begin{gather*}
        \int\limits^T_0 H(\varphi)\psi{dt} = \int\limits^T_0 H(\varphi)(-\pa_t H^*(\psi) + \mathcal{B}_n H^*(\psi)){dt}=\\
        = \int\limits^T_0 \mathcal{B}_n H(\varphi)H^*(\psi){dt} - H(\varphi) H^*(\varphi)\Big\vert^T_0 + \int\limits^T_0 \pa_t H(\varphi) H^*(\psi){dt} = \\
        = \int\limits^T_0 (\pa_t H(\varphi) + \mathcal{B}_n H(\varphi))H^*(\psi){dt} = \int\limits^T_0 \varphi H^*(\psi){dt}.
    \end{gather*}
    Thus, we have
    \begin{equation}\label{H adj rel}
        \int\limits^T_0 H(\varphi) \psi{dt} = \int\limits^T_0 \varphi H^*(\psi){dt}.
    \end{equation}
\end{remark}
\begin{remark}\label{H explicit solution}
    The solutions to the problems \eqref{H prob} and \eqref{adj H prob} can be found explicitly using the standard methods. We have
    \begin{equation}\label{eq: H explicit}
        H(\varphi)(x, t) = \int\limits^t_0 e^{-\mathcal{B}_n(t - \tau)}\varphi(x, \tau){d\tau},\quad
        H^*(\varphi)(x, t) = \int\limits^T_{t}e^{\mathcal{B}_n(t - \tau)} \varphi(x, \tau) {d\tau}.
    \end{equation}
    This explains the non-local in time nature of the \textquotedblleft strange
terms\textquotedblright\ $H(u_0)(x,t)$ and $H^{\ast }(p_0)(x,t)$
arising in the homogenized system (\ref{limit coupled prob}).
\end{remark}

The pair of functions $(u_{0},p_{0})$ can be used now to characterize the
the optimal control problem is given by the homogenized state problem (\ref%
{limit u prob}) and a suitable limit cost functional $J_{0}(v)$ which we
will show to be given by expression (\ref{limit cost functional}).

\begin{theorem}\label{thm: cost func limit} Under the same conditions as in Theorem~\ref%
{thm: main theorem}, we have 
\begin{equation*}
\lim\limits_{\varepsilon \rightarrow 0}J_{\varepsilon }(v_{\varepsilon
})=J_{0}(v_{0}),
\end{equation*}%
where $v_{\varepsilon }$ is the optimal control of the problem 
\eqref{init
state prob}, \eqref{cost functional}, and $v_{0}$ is the optimal control of
the problem \eqref{limit u prob}, \eqref{limit cost functional}. 
\end{theorem}

\begin{remark}
    The optimal control $v_0$ of the problem \eqref{limit u prob}, \eqref{limit cost functional} is characterized by the coupled system \eqref{limit coupled prob} and we have the relation $v_0 = -N^{-1}\chi_{\omega}p_0$. Thus, the theorem similar to Theorem~\ref{thm: init opt cont} is also valid.
\end{remark}


\begin{proof} We begin the proof of the Theorem~\ref{thm: main theorem} that is split into two parts. First, we derive effective equations for $u_0$ in Section~\ref{sec: char u}. Second, based on the derived problem for $u_0$, we find the limit problem for $p_0$ on Section~\ref{sec: char p}.

\section{Characterization of $u_0$}\label{sec: char u}
First, we define auxiliary functions $w^j_\ee$ as a solution to the boundary value problem
\begin{equation}\label{aux func prop}
    \left\{
        \begin{array}{lr}
            \Delta w^j_\ee = 0, & x\in T^j_{\ee/4}\setminus\overline{G^j_\ee}, \\
            w^j_\ee = 0, & x\in \pa T^j_{\ee/4}, \\
            w^j_\ee = 1, & x\in \pa G^j_\ee.
        \end{array}
    \right.
\end{equation}
The solution to \eqref{aux func prop} is given explicitly by
\begin{equation}\label{aux func explicit}
    w^j_\ee = \frac{|x - P^j_\ee|^{2 - n} - (\ee/4)^{2 - n}}{a_\ee^{2 - n} - (\ee/4)^{2 - n}}.
\end{equation}
Next, we define
\begin{equation*}
    W_\ee = \left\{
        \begin{array}{lr}
            w^j_\ee, & x \in T^j_{\ee/4}\setminus\overline{G^j_\ee},\, j\in\Upsilon_\ee, \\
            1, & x \in G_\ee,\\
            0, & x \in \mathbb{R}^n\setminus \overline{\bigcup_{j\in\Upsilon_\ee}T^j_{\ee/4} }.
        \end{array}
    \right.
\end{equation*}
It is easy to see that $W_\ee\in H^1_0(\Omega, \pa\Omega)$ and $W_\ee \rightharpoonup 0$ weakly in $H^1_0(\Omega)$ as $\ee\to 0$. The embedding theorem implies that $W_\ee \to 0$ strongly in $L^2(\Omega)$.


We take $W_\ee H^*(\varphi)$, where $\varphi = \psi(x) \eta(t)$ with $\psi(x)\in C^\infty_0(\Omega)$, $\eta(t) \in C^1([0, T])$, as a test function in the integral identity \eqref{int ident u} and get
\begin{gather*}
    \int\limits_{Q^T_\ee}\pa_t u_\ee W_\ee H^*(\varphi){dxdt} + \ee^{-\gamma}\int\limits_{S^T_\ee}\pa_t u_\ee H^*(\varphi){dsdt} + \int\limits_{Q^T_\ee}\nabla u_\ee \nabla(W_\ee H^*(\varphi)){dxdt} = \\
    = \int\limits_{Q^T_\ee}(f - N^{-1}\chi_{\omega_\ee}p_\ee)W_\ee H^*(\varphi){dxdt}.
\end{gather*}

Using convergences \eqref{limit func def} and the properties of $W_\ee$, we have
\begin{gather*}
    \lim\limits_{\ee\to0}\int\limits_{Q^T_\ee}\pa_t u_\ee W_\ee H^*(\varphi){dxdt} = 0,\\
    \lim\limits_{\ee\to0}\int\limits_{Q^T_\ee}(f - N^{-1}\chi_{\omega_\ee}p_\ee)W_\ee H^*(\varphi) {dxdt} = 0,\\
    \lim\limits_{\ee\to 0}\int\limits_{Q^T_\ee}\nabla u_\ee\nabla(W_\ee H^*(\varphi)){dxdt} = \lim\limits_{\ee\to0}\int\limits_{Q^T_\ee}\nabla(u_\ee H^*(\varphi))\nabla W_\ee{dxdt}.
\end{gather*}
Thus, we get
\begin{equation*}
    \ee^{-\gamma}\int\limits_{S^T_\ee}\pa_t u_\ee H^*(\varphi){dsdt} = -\int\limits_{Q^T_\ee}\nabla W_\ee\nabla(u_\ee H^*(\varphi)){dxdt} + \alpha_{1, \ee},
\end{equation*}
where $\alpha_{1, \ee} \to 0$ as $\ee\to 0$. Then, we make the decomposition
\begin{gather*}
    \int\limits_{Q^T_\ee}\nabla W_\ee\nabla (u_\ee H^*(\varphi)){dxdt} = \\
    = \sum\limits_{j\in\Upsilon_\ee}\int\limits^T_0\int\limits_{\pa T^j_{\ee/4}}\pa_\nu w^j_\ee u_\ee H^*(\varphi){dsdt} +\sum\limits_{j\in\Upsilon_\ee}\int\limits^T_0\int\limits_{\pa G^j_\ee}\pa_\nu w^j_\ee u_\ee H^*(\varphi){dsdt} = \\
    = -\ee C^{n - 2}_0(n - 2) 4^{n - 1}\sum\limits_{j\in\Upsilon_\ee}\int\limits^T_0\int\limits_{\pa T^j_{\ee/4}}u_\ee H^*(\varphi){dsdt} 
    + \mathcal{B}_n\ee^{-\gamma}\int\limits_{S^T_\ee}u_\ee H^*(\varphi){dsdt} + \alpha_{2, \ee},
\end{gather*}
where $\alpha_{2, \ee} \to 0$ as $\ee\to 0$. For the first integral, we use Lemma~1 from \cite{ZuShDiffEq} to get
\begin{equation*}
    \lim\limits_{\ee\to0}\ee C^{n-2}_0(n-2) 4^{n - 1} \sum\limits_{j\in\Upsilon_\ee}\int\limits^T_0\int\limits_{\pa T^j_{\ee/4}}u_\ee H^*(\varphi){dsdt} =\mathcal{A}_n \int\limits_{Q^T}u_0 H^*(\varphi){dxdt} = \mathcal{A}_n\int\limits_{Q^T}H(u_0) \varphi{dxdt},
\end{equation*}
where the last equality is due to \eqref{H adj rel}. Thus, we derive
\begin{gather}\label{u S limit 1}
    \lim\limits_{\ee\to 0}\ee^{-\gamma}\int\limits_{S^T_\ee}\pa_t u_\ee H^*(\varphi){dsdt} = \mathcal{A}_n \int\limits_{Q^T}H(u_0) \varphi{dxdt} - \lim\limits_{\ee\to0}\ee^{-\gamma}\int\limits_{S^T_\ee}\mathcal{B}_n u_\ee H^*(\varphi){dsdt} 
\end{gather}
Then, we have
\begin{gather}
    \nonumber\int\limits_{S^T_\ee}\pa_t u_\ee H^*(\varphi){dsdt} = \int\limits_{S_\ee}(u_\ee(x, T) H^*(\varphi)(x, T) - u_\ee(x, 0) H^*(\varphi)(x, 0)){ds} -\\ \label{H u change}
    - \int\limits_{S^T_\ee}u_\ee \pa_t H^*(\varphi){dsdt} 
    = - \int\limits_{S^T_\ee}u_\ee \pa_t H^*(\varphi){dsdt},
\end{gather}
the last equality follows from the fact that $u_\ee(x, 0) = 0$ and $H^*(\varphi)(x, T) = 0$ for $x\in S_\ee$.

Thus, from \eqref{u S limit 1}, \eqref{H u change}, we conclude that
\begin{equation*}
    \lim\limits_{\ee\to 0}\ee^{-\gamma}\int\limits_{S^T_\ee}u_\ee (-\pa_t H^*(\varphi) + \mathcal{B}_n H^*(\varphi)){dsdt} = \mathcal{A}_n\int\limits_{Q^T}H(u_0) \varphi{dxdt}.
\end{equation*}
Using the definition of $H^*(\varphi)$, we get
\begin{equation}\label{u S limit}
    \lim\limits_{\ee\to 0}\ee^{-\gamma}\int\limits_{S^T_\ee}u_\ee \varphi{dsdt} = \mathcal{A}_n\int\limits_{Q^T}H(u_0) \varphi{dxdt}.
\end{equation}

Now, we take $\varphi W_\ee$ in the integral identity \eqref{int ident u} and, similarly to the above, using \eqref{u S limit}, we derive
\begin{gather*}
    \lim\limits_{\ee\to 0}\ee^{-\gamma}\int\limits_{S^T_\ee} \pa_t u_\ee \varphi{dsdt} = \mathcal{A}_n\int\limits_{Q^T}u_0 \varphi{dxdt} - \lim\limits_{\ee\to0}\ee^{-\gamma}\mathcal{B}_n\int\limits_{S^T_\ee}u_\ee \varphi{dsdt} = \\
    = \mathcal{A}_n\int\limits_{Q^T}u_0 \varphi{dxdt} - \mathcal{A}_n\mathcal{B}_n\int\limits_{Q^T}H(u_0)\varphi{dxdt} = \mathcal{A}_n\int\limits_{Q^T}(u_0 - \mathcal{B}_n H(u_0))\varphi{dxdt}.
\end{gather*}
Now, we can pass to the limit as $\ee\to 0$ in the integral identity \eqref{int ident u}. Doing so, we conclude that $u_0$ satisfies the integral identity
\begin{gather*}
    \int\limits_{Q^T}\pa_t u_0 v{dxdt} + \int\limits_{Q^T}\nabla u_0 \nabla v{dxdt} 
    + \mathcal{A}_n\int\limits_{Q^T}(u_0 - \mathcal{B}_n H(u_0))v{dxdt} = \int\limits_{Q^T}(f - N^{-1}\chi_{\omega}p_0)v{dxdt}.
\end{gather*}

\section{Characterization of $p_0$}\label{sec: char p}
We take $\varphi = W_\ee H(\varphi)$, where $\varphi = \psi(x) \eta(t)$ with $\psi\in C^\infty_0(\Omega)$, $\eta\in C^1([0, T])$ as a test function in the integral identity \eqref{int ident init adjoint prob} and get
\begin{gather*}
    \int\limits_{Q^T_\ee}\nabla W_\ee \nabla(p_\ee H(\varphi)){dxdt} -\ee^{-\gamma}\int\limits_{S^T_\ee}\pa_t(p_\ee - u_\ee) H(\varphi){dsdt} =\\ 
    = \int\limits_{Q^T_\ee}\pa_t(p_\ee - u_\ee)W_\ee H(\varphi){dxdt} + \int\limits_{Q^T_\ee}(f - N^{-1}\chi_{\omega_\ee}p_\ee)W_\ee H(\varphi){dxdt} + \beta_\ee = \kappa_\ee, 
\end{gather*}
where $\beta_\ee \to 0$ as $\ee\to 0$. Due to the properties of $u_\ee$, $p_\ee$, $W_\ee$, we have that $\kappa_\ee \to 0$ as $\ee\to0$. Hence, by decomposing the first integral, we derive
\begin{gather}
    \ee^{-\gamma}\mathcal{B}_n\int\limits_{S^T_\ee} p_\ee H(\varphi){dsdt} + \sum\limits_{j\in\Upsilon_\ee}\int\limits^{T}_{0}\int\limits_{\pa T^j_{\ee/4}}\pa_\nu w^j_\ee p_\ee H(\varphi){dsdt} 
    \label{p decomp 1}
    - \ee^{-\gamma}\int\limits_{S^T_\ee}\pa_t(p_\ee - u_\ee) H(\varphi){dsdt} = \kappa_{1, \ee},
\end{gather}
where $\kappa_{1, \ee}\to 0$ as $\ee\to 0$. For the last integral, we have
\begin{equation*}
\begin{gathered}
    \ee^{-\gamma}\int\limits_{S^T_\ee}\pa_t(p_\ee - u_\ee)H(\varphi){dsdt}  
    = \ee^{-\gamma}\int\limits_{S^T_\ee}\pa_t(p_\ee - u_\ee + u_T)H(\varphi){dsdt} -\ee^{-\gamma}\int\limits_{S^T_\ee}\pa_t u_T H(\varphi){dsdt} = \\
    = \ee^{-\gamma}\int\limits_{S_\ee} (p_\ee - u_\ee + u_T)(x, T)H(\varphi)(x, T){ds} 
    - \ee^{-\gamma}\int\limits_{S_\ee} (p_\ee - u_\ee + u_T)(x, 0)H(\varphi)(x, 0){ds} - \\
    - \ee^{-\gamma}\int\limits_{S^T_\ee}\pa_t H(\varphi)(p_\ee - u_\ee + u_T){dsdt} - \ee^{-\gamma}\int\limits_{S^T_\ee}\pa_t u_T H(\varphi){dsdt}.
\end{gathered}
\end{equation*}
From the definitions of $H(\varphi)$ and $p_\ee$, we have $H(\varphi)(x, 0) \equiv 0$ and $p_\ee(x, T) - u_\ee(x,T) + u_T(x, T) = 0$, hence, we obtain
\begin{equation*}
\begin{gathered}
   \ee^{-\gamma}\int\limits_{S^T_\ee}\pa_t(p_\ee - u_\ee)H(\varphi){dsdt} = - \ee^{-\gamma}\int\limits_{S^T_\ee}\pa_t u_T H(\varphi){dsdt} - \ee^{-\gamma}\int\limits_{S^T_\ee}\pa_t H(\varphi)(p_\ee - u_\ee + u_T){dsdt}.
\end{gathered}
\end{equation*}
Thus, substituting this expression into \eqref{p decomp 1}, we get
\begin{gather*}
    \ee^{-\gamma}\int\limits_{S^T_\ee}p_\ee(\mathcal{B}_n H(\varphi) + \pa_t H(\varphi)){dsdt} = -\sum\limits_{j\in\Upsilon_\ee}\int\limits^{T}_{0}\int\limits_{\pa T^j_{\ee/4}}\pa_\nu w^j_\ee p_\ee H(\varphi){dsdt} - \\
    - \ee^{-\gamma}\int\limits_{S^T_\ee} (\pa_t u_T H(\varphi) + \pa_t H(\varphi) u_T){dsdt} + \ee^{-\gamma}\int\limits_{S^T_\ee}\pa_t H(\varphi) u_\ee{dsdt} + \beta_{1,\ee},
\end{gather*}
where $\beta_{1, \ee} \to 0$ as $\ee\to0$. Using the definition of $H(\varphi)$, we obtain
\begin{gather}
    \ee^{-\gamma}\int\limits_{S^T_\ee}p_\ee \varphi{dsdt} =-\sum\limits_{j\in\Upsilon_\ee}\int\limits^{T}_{0}\int\limits_{\pa T^j_{\ee/4}}\pa_\nu w^j_\ee p_\ee H(\varphi){dsdt} - \nonumber\\ \label{p int limit}
    - \ee^{-\gamma}\int\limits_{S^T_\ee} (\pa_t u_T H(\varphi) + \pa_t H(\varphi) u_T){dsdt} + \ee^{-\gamma}\int\limits_{S^T_\ee}(\varphi - \mathcal{B}_n H(\varphi))u_\ee{dsdt} + \beta_{1, \ee}.
\end{gather}
The limit of the first integral on the equality's right-hand side can be found using Lemma 1 from \cite{ZuShDiffEq}. For the third integral, we use the convergence for $u_\ee$ obtained above. Let us find the limit of the second integral. Using the definition of $H$, we have
\begin{gather}
    \ee^{-\gamma}\int\limits_{S^T_\ee} (\pa_t u_T H(\varphi) + \pa_t H(\varphi) u_T){dsdt}
    = \ee^{-\gamma}\int\limits_{S^T_\ee}(\pa_t u_T H(\varphi) + (\varphi - \mathcal{B}_n H(\varphi))u_T){dsdt} = \nonumber \\
    = \ee^{-\gamma}\int\limits_{S^T_\ee}(\pa_t u_T - \mathcal{B}_n u_T)H(\varphi){dsdt} + \ee^{-\gamma}\int\limits_{S^T_\ee}u_T\varphi{dsdt} \label{p H limit int 1} 
    = \ee^{-\gamma}\int\limits_{S^T_\ee}(H^*(\pa_t u_T - \mathcal{B}_n u_T) + u_T)\varphi{dsdt}.
\end{gather}
Using \eqref{eq: H explicit}, we get
\begin{gather*}
    H^*(\pa_t u_T - \mathcal{B}_n u_T)(x, t) = \int\limits^T_t e^{\mathcal{B}_n(t - \tau)}(\pa_t u_T - \mathcal{B}_n u_T){d\tau} = \\
    = e^{\mathcal{B}_n(t - \tau)} u_T \vert^T_t + \int\limits^T_t \mathcal{B}_n e^{\mathcal{B}_n(t - \tau)} u_T{d\tau} - \int\limits^T_t \mathcal{B}_n e^{\mathcal{B}_n(t - \tau)} u_T{d\tau} 
    = e^{\mathcal{B}_n (t - T)}u_T(x, T) - u_T(x, t).
\end{gather*}
From here, it follows that
\begin{equation}\label{p H limit int 2}
    \ee^{-\gamma}\int\limits_{S^T_\ee}(H^*(\pa_t u_T - \mathcal{B}_n u_T) + u_T)\varphi{dsdt} = \ee^{-\gamma}\int\limits_{S^T_\ee}e^{\mathcal{B}_n (t - T)}u_T(x, T)\varphi(x, t){dsdt}.
\end{equation}

Using the regularity of the function $\varphi$ and $u_T$, we get the convergence
\begin{equation}\label{p H limit int 3}
    \lim\limits_{\ee\to 0}\ee^{-\gamma}\int\limits_{S^T_\ee}e^{\mathcal{B}_n (t - T)}u_T(x, T)\varphi{dsdt} = C_0^{n-1}\omega_n\int\limits_{Q^T}e^{\mathcal{B}_n (t - T)} u_T(x, T)\varphi(x, t){dxdt}.
\end{equation}

Using Lemma 1 from \cite{ZuShDiffEq} and \eqref{p H limit int 1}-\eqref{p H limit int 3}, from \eqref{p int limit}, we have
\begin{gather*}
    \lim\limits_{\ee\to 0}\ee^{-\gamma}\int\limits_{S^T_\ee}p_\ee \varphi{dsdt}
    =\mathcal{A}_n\int\limits_{Q^T}p_0 H(\varphi){dxdt} +  \\ 
    + \lim\limits_{\ee\to 0}\ee^{-\gamma}\int\limits_{S^T_\ee}(\varphi - \mathcal{B}_n H(\varphi))u_\ee{dsdt} - C^{n-1}_0\omega_n\int\limits_{Q^T}e^{\mathcal{B}_n(t - T)}u_T(x, T)\varphi{dxdt}.
\end{gather*}
For the second integral on the right-hand side of the expression above, we use the results of the previous section and get
\begin{gather*}
    \lim\limits_{\ee\to 0}\ee^{-\gamma}\int\limits_{S^T_\ee}p_\ee \varphi{dsdt} 
    = \mathcal{A}_n \int\limits_{Q^T}p_0 H(\varphi){dxdt} + \\ + \mathcal{A}_n\int\limits_{Q^T}(H(u_0)\varphi - \mathcal{B}_n H(u_0)H(\varphi)){dxdt} 
    - C^{n-1}_0\omega_n\int\limits_{Q^T}e^{\mathcal{B}_n(t - T)}u_T(x, T)\varphi{dxdt} = \\
    = \mathcal{A}_n \int\limits_{Q^T}H^*(p_0) \varphi{dxdt} + \mathcal{A}_n\int\limits_{Q^T}(H(u_0) - \mathcal{B}_n H^*(H(u_0)))\varphi{dxdt} 
    - C^{n-1}_0\omega_n\int\limits_{Q^T}e^{\mathcal{B}_n(t - T)}u_T(x, T)\varphi{dxdt}.
\end{gather*}
Again, we take $\varphi W_\ee$ as the test function in the integral identity \eqref{int ident init adjoint prob} and derive
\begin{gather*}
    -\ee^{-\gamma}\int\limits_{S^T_\ee}\pa_t p_\ee \varphi {dsdt} = -\ee^{-\gamma}\mathcal{B}_n\int\limits_{S^T_\ee}p_\ee \varphi {dsdt} - \\
    - \sum\limits_{j\in\Upsilon_\ee}\int\limits^T_0\int\limits_{\pa T^j_{\ee/4}}\pa_\nu w^j_\ee p_\ee \varphi{dsdt} - \ee^{-\gamma}\int\limits_{S^T_\ee}\pa_t u_\ee \varphi {dsdt} + \alpha_\ee,
\end{gather*}
where $\alpha_\ee \to 0$ as $\ee\to 0$. From here, we have (note that $C^{n-1}_0\omega_n \mathcal{B}_n = \mathcal{A}_n$)
\begin{gather*}
    \lim\limits_{\ee\to 0}-\ee^{-\gamma}\int\limits_{S^T_\ee}\pa_t p_\ee\varphi{dsdt} = 
    -\mathcal{A}_n\mathcal{B}_n\int\limits_{Q^T}(H^*(p_0) + H(u_0) - \mathcal{B}_n H^*(H(u_0)))\varphi{dxdt} + \\
    + \mathcal{A}_n\int\limits_{Q^T}e^{\mathcal{B}_n(t - T)}u_T(x, T)\varphi {dxdt} + \mathcal{A}_n \int\limits_{Q^T}p_0\varphi{dxdt} - \mathcal{A}_n\int\limits_{Q^T}(u_0 - H(u_0))\varphi{dxdt} = \\
    = \mathcal{A}_n\int\limits_{Q^T}((p_0 - \mathcal{B}_n H^*(p_0)) - (u_0 - \mathcal{B}_n H^*(H(u_0))))\varphi{dxdt} 
    + \mathcal{A}_n\int\limits_{Q^T}e^{\mathcal{B}_n(t - T)}u_T(x, T)\varphi {dxdt}.
\end{gather*}
Thus, passing to the limit in the integral identity \eqref{int ident init adjoint prob}, we conclude that $p_0$ satisfies the integral identity
\begin{equation}\label{limit p int ident}
    \begin{gathered}
        -\int\limits_{Q^T}\pa_t p_0 \varphi{dxdt} + \int\limits_{Q^T}\nabla p_0 \nabla\varphi{dxdt} + \\ 
        + \mathcal{A}_n\int\limits_{Q^T}(p_0 - \mathcal{B}_n H^*(p_0))\varphi{dxdt} + \mathcal{A}_n\int\limits_{Q^T}e^{\mathcal{B}_n(t - T)}u_T(x, T)\varphi {dxdt} =\\
        = \int\limits_{Q^T}\nabla (u_0 - u_T) \nabla \varphi{dxdt} + \mathcal{A}_n \int\limits_{Q^T}(u_0 - \mathcal{B}_n^2 H^*(H(u_0)))\varphi{dxdt}.
    \end{gathered}
\end{equation}
This concludes the proof of the Theorem~\ref{thm: main theorem}.
\end{proof}

\section{Cost functional limit}
\begin{proof} Here, we give the proof of Theorem~\ref{thm: cost func limit}.

Let us find the limit of the cost functional $J_\ee$ as $\ee\to 0$. As $v_\ee = N^{-1}p_\ee$, we have
\begin{gather*}
    J_\ee(v_\ee) = \frac{1}{2}\int\limits_{Q^T_\ee}|\nabla u_\ee - u_T|^2{dxdt} + \frac{1}{2}\int\limits_{\Omega_\ee}(u_\ee - u_T)^2(x, T){dx} + \\
    + \frac{\ee^{-\gamma}}{2}\int\limits_{S_\ee}(u_\ee - u_T)^2(x, T){ds} 
    + \frac{1}{2N}\int\limits_{\omega^T_\ee}p_\ee^2{dxdt}.
\end{gather*}
Using integral identity \eqref{int ident u}, we transform the functional $J_\ee$ and get
\begin{gather*}
    J_\ee(v_\ee) = \frac{1}{2}\int\limits_{Q^T_\ee}|\nabla(u_\ee - u_T)|^2{dxdt} + \frac{1}{2}\int\limits_{\Omega_\ee}(u_\ee - u_T)^2(x, T){dx} + \\
    + \frac{\ee^{-\gamma}}{2}\int\limits_{S_\ee}(u_\ee - u_T)^2(x, T){ds} + \frac{1}{2}\int\limits_{Q^T_\ee}f p_\ee{dxdt} - \\ 
    - \frac{1}{2}\int\limits_{Q^T_\ee}\nabla u_\ee \nabla p_\ee {dxdt} 
    - \frac{1}{2}\int\limits_{Q^T_\ee}\pa_t u_\ee p_\ee{dxdt} - \frac{\ee^{-\gamma}}{2}\int\limits_{S^T_\ee}\pa_t u_\ee p_\ee{dsdt}.
\end{gather*}
Next, we have
\begin{gather*}
    - \frac{1}{2}\int\limits_{Q^T_\ee}\nabla u_\ee \nabla p_\ee {dxdt} 
    - \frac{1}{2}\int\limits_{Q^T_\ee}\pa_t u_\ee p_\ee{dxdt} - \frac{\ee^{-\gamma}}{2}\int\limits_{S^T_\ee}\pa_t u_\ee p_\ee{dsdt} =\\
    = - \frac{1}{2}\int\limits_{Q^T_\ee}\nabla u_\ee \nabla p_\ee {dxdt} 
    + \frac{1}{2}\int\limits_{Q^T_\ee}\pa_t p_\ee u_\ee{dxdt} + \frac{\ee^{-\gamma}}{2}\int\limits_{S^T_\ee}\pa_t p_\ee u_\ee{dsdt} - \\
    - \frac{1}{2}\int\limits_{\Omega_\ee}u_\ee(x, T)(u_\ee - u_T)(x, T){dx} - \frac{\ee^{-\gamma}}{2}\int\limits_{S_\ee}u_\ee(x, T)(u_\ee - u_T)(x, T){ds} = \\
    = -\frac{1}{2}\int\limits_{Q^T_\ee}\nabla(u_\ee - u_T)\nabla u_\ee{dxdt} - \frac{1}{2}\int\limits_{\Omega_\ee}u_\ee(x, T)(u_\ee - u_T)(x, T){dx} - \\
    - \frac{\ee^{-\gamma}}{2}\int\limits_{S_\ee}u_\ee(x, T)(u_\ee - u_T)(x, T){ds}.
\end{gather*}
Thus, we can further transform the expression for the cost functional
\begin{gather*}
    J_\ee(v_\ee) = \frac{1}{2}\int\limits_{Q^T_\ee}\nabla u_T\nabla (u_T - u_\ee){dxdt} + \frac{1}{2}\int\limits_{\Omega_\ee}u_T(x, T)(u_T - u_\ee)(x, T){dx} +\\
    +\frac{\ee^{-\gamma}}{2}\int\limits_{S_\ee}u_T(x, T) (u_T - u_\ee)(x, T){ds} + \frac{1}{2}\int\limits_{Q^T_\ee}fp_\ee{dxdt} = \\
    = \frac{1}{2}\int\limits_{Q^T_\ee}\nabla u_T\nabla (u_T - u_\ee){dxdt} + \frac{1}{2}\int\limits_{Q^T_\ee}f p_\ee{dxdt} + \\
    + \frac{1}{2}\int\limits_{\Omega_\ee}u^2_T(x, T){dx} - \frac{1}{2}\int\limits_{Q^T_\ee}(\pa_t u_\ee u_T + \pa_t u_T u_\ee){dxdt} + \\
    + \frac{\ee^{-\gamma}}{2}\int\limits_{S_\ee}u^2_T(x, T){ds} - \frac{\ee^{-\gamma}}{2}\int\limits_{S^T_\ee}(\pa_t u_\ee u_T + \pa_t u_T u_\ee){dxdt}.
\end{gather*}
Now, we should pass to the limit in the obtained expressions. First, properties of $u_\ee$, $\pa_t u_\ee$ and $p_\ee$ imply that, as $\ee\to 0$,
\begin{gather*}
    \int\limits_{Q^T_\ee}\nabla u_T \nabla(u_T - u_\ee){dxdt} \to \int\limits_{Q^T}\nabla u_T \nabla(u_T - u_0){dxdt},\\
    \int\limits_{Q^T_\ee}fp_\ee{dxdt} \to \int\limits_{Q^T}fp_0{dxdt},\\
    \int\limits_{Q^T_\ee}(\pa_t u_\ee u_T + \pa_t u_T u_\ee){dxdt} \to \int\limits_{Q^T}(\pa_t u_0 u_T + \pa_t u_T u_0){dxdt},\\
    \int\limits_{\Omega_\ee}u^2_T(x, T){dx} \to \int\limits_{\Omega}u^2_T(x, T){dx}.
\end{gather*}
Next, we find the limit of the integrals over $S_\ee$ based on the convergences obtained in the previous sections
\begin{gather*}
    \lim\limits_{\ee\to0}\ee^{-\gamma}\int\limits_{S^T_\ee}\pa_t u_T u_\ee{dsdt} = \mathcal{A}_n\int\limits_{Q^T}H(u_0)\pa_t u_T {dxdt},\\
    \lim\limits_{\ee\to0}\ee^{-\gamma}\int\limits_{S^T_\ee}\pa_t u_\ee u_T {dsdt} = \mathcal{A}_n\int\limits_{Q^T}(u_0 - \mathcal{B}_n H(u_0))u_T{dxdt},\\
    \lim\limits_{\ee\to0}\ee^{-\gamma}\int\limits_{S_\ee}u^2_T(x, T){ds} = C^{n-1}_0 \omega_n\int\limits_{\Omega}u^2_T(x, T){dx}.
\end{gather*}
Combining all of the above convergences, we derive
\begin{gather}
    \lim\limits_{\ee\to0}J_\ee(v_\ee) = \frac{1}{2}\int\limits_{Q^T}\nabla u_T\nabla(u_T - u_0){dxdt} + \frac{1}{2}\int\limits_{Q^T}fp_0{dxdt} + \nonumber \\
    + \frac{1}{2}\int\limits_{\Omega} u^2_T(x, T){dx} + \frac{C^{n-1}_0 \omega_n}{2}\int\limits_{\Omega}u^2_T(x, T){dx} - \int\limits_{Q^T}(\pa_t u_0 u_T + \pa_t u_T u_0){dxdt} - \nonumber \\ \label{cost limit}
    - \frac{\mathcal{A}_n}{2}\int\limits_{Q^T}H(u_0)\pa_t u_T{dxdt} - \frac{\mathcal{A}_{n}}{2}\int\limits_{Q^T}(u_0 - \mathcal{B}_n H(u_0))u_T{dxdt}.
\end{gather}
Again, we have
\begin{equation}\label{cost limit 1}
    \frac{1}{2}\int\limits_{\Omega} u^2_T(x, T){dx} - \frac{1}{2}\int\limits_{Q^T}(\pa_t u_0 u_T + \pa_t u_T u_0){dxdt} 
    = \frac{1}{2}\int\limits_{\Omega} u_T(x, T)(u_T - u_0)(x, T){dx}.
\end{equation}
Also, using the definition of $H$, we get
\begin{gather}
    \frac{\mathcal{A}_n}{2}\int\limits_{Q^T}H(u_0)\pa_t u_T{dxdt} + \frac{\mathcal{A}_{n}}{2}\int\limits_{Q^T}(u_0 - \mathcal{B}_n H(u_0))u_T{dxdt} = \nonumber \\
    = \frac{\mathcal{A}_n}{2}\int\limits_{Q^T}(\pa_t u_T H(u_0) + \pa_t H(u_0) u_T){dxdt} = \frac{\mathcal{A}_n}{2}\int\limits_{Q^T}\pa_t (u_T H(u_0)){dxdt} = \nonumber \\ \label{cost limit 2}
    = \frac{\mathcal{A}_n}{2}\int\limits_{\Omega}u_T(x, T) H(u_0)(x, T){dx}.
\end{gather}

From the integral identities for $u_0$ and $p_0$, we get
\begin{gather}
    \frac{1}{2}\int\limits_{Q^T} f p_0{dxdt} = \frac{1}{2}\int\limits_{Q^T}\pa_t u_0 p_0{dxdt} + \frac{1}{2}\int\limits_{Q^T}\nabla u_0 \nabla p_0 {dxdt} + \nonumber \\
    + \frac{\mathcal{A}_n}{2}\int\limits_{Q^T}(u_0 - \mathcal{B}_nH(u_0))p_0 {dxdt}
    + \frac{1}{2N}\int\limits_{\omega^T}p^2_0{dxdt} = \nonumber \\
    = \frac{1}{2}\int\limits_{\Omega}u_0(x, T)(u_0 - u_T)(x, T){dx} - \frac{1}{2}\int\limits_{Q^T}\pa_t p_0 u_0 {dxdt} + \frac{1}{2}\int\limits_{Q^T}\nabla u_0 \nabla p_0{dxdt} + \nonumber \\
    + \frac{\mathcal{A}_n}{2}\int\limits_{Q^T}(p_0 - \mathcal{B}_nH^*(p_0))u_0 {dxdt}
    + \frac{1}{2N}\int\limits_{\omega^T}p^2_0{dxdt} = \nonumber \\
    = \frac{1}{2}\int\limits_{\Omega}u_0(x, T)(u_0 - u_T)(x, T){dx} + \frac{1}{2}\int\limits_{Q^T}\nabla u_0 \nabla(u_0 - u_T){dxdt} + \frac{1}{2N}\int\limits_{\omega^T}p^2_0{dxdt} + \nonumber \\ \label{cost limit 3}
    + \frac{\mathcal{A}_n}{2}\int\limits_{Q^T}(u_0 - \mathcal{B}_n^2H^*(H(u_0)))u_0{dxdt} - \frac{\mathcal{A}_n}{2}\int\limits_{Q^T}e^{\mathcal{B}_n (t - T)}u_T(x, T) u_0{dxdt}.
\end{gather}
Next, we substitute \eqref{cost limit 1}-\eqref{cost limit 3} into \eqref{cost limit} and derive
\begin{gather}
    \lim\limits_{\ee\to 0}J_\ee(v_\ee) = \frac{1}{2}\Vert\nabla (u_0 - u_T)\Vert^2_{L^2(Q^T)} + \frac{1}{2}\Vert (u_0 - u_T)(x, T)\Vert^2_{L^2(\Omega)} + \frac{1}{2N}\Vert p_0\Vert^2_{L^2(\omega^T)} + \nonumber \\
    + \frac{C^{n-1}_0\omega_n}{2}\Vert u_T(x, T)\Vert^2_{L^2(\Omega)} + \frac{\mathcal{A}_n}{2}\int\limits_{Q^T}(u^2_0 - \mathcal{B}_n^2 H^2(u_0)){dxdt} - \nonumber \\ \label{cost limit final decomp}
    - \frac{\mathcal{A}_n}{2}\int\limits_{Q^T}e^{\mathcal{B}_n (t - T)}u_T(x, T) u_0{dxdt} -\frac{\mathcal{A}_n}{2}\int\limits_{\Omega}H(u_0)(x, T) u_T(x, T){dx}.
\end{gather}
Using the definition of $H$ and $H^*$, we make the transform
\begin{gather*}
    \int\limits_{Q^T}e^{\mathcal{B}_n(t - T)}u_T(x, T) u_0 {dxdt} = \int\limits_{Q^T}(H^*(\pa_t u_T - \mathcal{B}_n u_T)(x, t) + u_T(x, t))u_0{dxdt} = \\
    = \int\limits_{Q^T}(\pa_t u_T - \mathcal{B}_n u_T)H(u_0){dxdt} + \int\limits_{Q^T}u_T u_0{dxdt} = \\
    = \int\limits_{Q^T}(\pa_t u_T H(u_0) + (u_0 - \mathcal{B}_n H(u_0))u_T){dxdt} = \int\limits_{Q^T}(\pa_t(u_T H(u_0))){dxdt} = \\
    = \int\limits_{\Omega}u_T(x, T) H(u_0)(x, T){dxdt}.
\end{gather*}
Substituting this into \eqref{cost limit final decomp}, we get
\begin{gather*}
    \lim\limits_{\ee\to 0}J_\ee(v_\ee) = \frac{1}{2}\Vert\nabla (u_0 - u_T)\Vert^2_{L^2(Q^T)} + \frac{1}{2}\Vert (u_0 - u_T)(x, T)\Vert^2_{L^2(\Omega)} + \frac{1}{2N}\Vert p_0\Vert^2_{L^2(\omega^T)} + \\
    + \frac{C^{n-1}_0\omega_n}{2}\Vert u_T(x, T)\Vert^2_{L^2(\Omega)} + \frac{\mathcal{A}_n}{2}\int\limits_{Q^T}(u^2_0 - \mathcal{B}_n^2 H^2(u_0)){dxdt} 
    - \frac{\mathcal{A}_n}{2}\int\limits_{\Omega}2u_T(x, T)H(u_0)(x, T){dx}.
\end{gather*}
Using the definition of $H(u_0)$, we derive
\begin{gather*}
    u^2_0 - \mathcal{B}_n^2 H^2(u_0) = (\pa_t H(u_0) + \mathcal{B}_n H(u_0))^2 - \mathcal{B}_n^2 H^2(u_0)=\\
    = (\pa_t H(u_0))^2 + 2\mathcal{B}_n \pa_t H(u_0) H(u_0).
\end{gather*}
 Hence, we have
\begin{equation*}
    \int\limits_{Q^T}(u^2_0 - \mathcal{B}_n^2 H^2(u_0)){dxdt} = \int\limits_{Q^T}(\pa_t H(u_0))^2{dxdt} + \mathcal{B}_n\int\limits_{\Omega}H^2(u_0)(x, T){dx}.
\end{equation*}
Thus, we further transform the limit of $J_\ee$ and get
\begin{gather*}
     \lim\limits_{\ee\to 0}J_\ee(v_\ee) = \frac{1}{2}\Vert\nabla (u_0 - u_T)\Vert^2_{L^2(Q^T)} + \frac{1}{2}\Vert (u_0 - u_T)(x, T)\Vert^2_{L^2(\Omega)} + \frac{1}{2N}\Vert p_0\Vert^2_{L^2(\omega^T)} + \\
     + \frac{C^{n-1}_0\omega_n}{2}\Vert(u_T - \mathcal{B}_n H(u_0))(x, T)\Vert^2_{L^2(\Omega)} + \frac{\mathcal{A}_n}{2}\Vert \pa_t H(u_0)\Vert^2_{L^2(Q^T)} \equiv J_0(v_0).
\end{gather*}
This concludes the proof.
\end{proof}

\begin{remark}
\label{Rm Controlability Limit problem} As indicated in Remark \ref{Rm
Controlability probl init}, the mentioned arguments by J.-L. Lions in \cite%
{Lions Malaga} (see also, e.g., Section 1.6 in the book \cite{Glow-Lions}),
allow to get some results \ on the \textquotedblleft approximate
controllability property\textquotedblright\ for solutions of problem (\ref%
{limit u prob}). Given $u_{T}\in H^{1}(0,T;H_{0}^{1}(\Omega )) \bigcap C(\overline{Q^T})$ and an
arbitrarily small $\delta >0,$ the \textquotedblleft approximate
controllability property\textquotedblright\ consists now in finding a
control $v_{\delta }\in L^{2}(\omega _{\varepsilon }^{T})$ such that $%
\left\Vert u_{\varepsilon }(\cdot,T)-u_{T}(\cdot,T)\right\Vert _{L^{2}(\Omega )}\leq
\delta $. We introduce a new parameter $\kappa >0$ in the cost functional%
\begin{equation}
\begin{gathered}
J_{0}^{\kappa }(v)=\frac{\kappa }{2}\Vert \nabla (u_{0}(v)-u_{T})\Vert
_{L^{2}(Q^{T})}^{2}+\frac{\kappa }{2}\Vert u_{0}(v)(\cdot,T)-u_{T}(\cdot,T)\Vert
_{L^{2}(\Omega )}^{2}+\frac{N}{2}\Vert v\Vert _{L^{2}(\omega ^{T})}^{2} + \\ 
+\frac{\kappa C^{n-1}_0\omega _{n}}{2}\Vert u_{T}(\cdot,T)-\mathcal{B}%
_{n}H(u_0(v))(\cdot,T)\Vert _{L^{2}(\Omega )}^{2}+\frac{\kappa \mathcal{A}_{n}}{2%
}\Vert \partial _{t}H(u_0(v))\Vert _{L^{2}(Q^{T})}^{2},%
\end{gathered}
\end{equation}%
and if we know a result on \textquotedblleft unique
continuation\textquotedblright\ for problem \eqref{limit u prob}, then, by
using some a priori estimates on the adjoint state $p_{0}(x,t)$, it can be
shown that such searched control $v_{\delta }$ can be found by considering
the set of optimal controls $v_{\kappa }\in L^{2}(\omega _{\varepsilon
}^{T}) $ associated to $J_{0}^{\kappa }(v)$ and by taking $v_{\delta
}=v_{\kappa }$ for $\kappa $ large enough. Again, the presence of the first,
fourth and fifth terms in $J_{0}^{\kappa }(v)$ leads to conclude that the
the associated state will also satisfy some additional properties: $\Vert \nabla
(u_{0} - u_{T})\Vert _{L^{2}(Q^{T})}\leq \delta ,$ $\Vert u_{T}(\cdot,T)-\mathcal{B%
}_{n}H(u_0)(\cdot,T)\Vert _{L^{2}(\Omega )}\leq \delta $ and $\Vert \partial
_{t}H(u_0)\Vert _{L^{2}(Q^{T})}^{2}\leq \delta .$
\end{remark}

\begin{center}
\textbf{Acknowledgments}
\end{center}

The research of J. I. D\'{\i}az was partially supported by the project
PID2020-112517 GB-I00 of the Spain State Research Agency (AEI).


\end{document}